\documentclass{article}
\usepackage[latin1]{inputenc}
\usepackage{amsmath}
\usepackage{amssymb}
\usepackage{amsthm}
\usepackage{graphicx}
\usepackage{enumerate}
\usepackage{color}
\usepackage{algorithm}
\usepackage{algorithmic}

\usepackage[sectionbib]{natbib}
\usepackage{bibentry}

\theoremstyle{plain}
\newtheorem{theorem}{Theorem}[subsection]
\newtheorem{lemma}{Lemma}[subsection]

\newtheorem{definition}{Definition}

\theoremstyle{definition}

\theoremstyle{remark}
\newtheorem{remark}{Remark}[subsection]
\newtheorem{example}{Example}[subsection]
\DeclareMathOperator*{\argmin}{arg\,min}

\newcommand{\eps}{\varepsilon}

\newcommand{\lan}{\left\langle}
\newcommand{\ran}{\right\rangle}

\newcommand{\bitem}{\begin{itemize}}
\newcommand{\eitem}{\end{itemize}}

\newcommand{\bpm}{\begin{pmatrix}}
\newcommand{\epm}{\end{pmatrix}}
\newcommand{\T}{\top}

\newcommand{\la}{\langle}
\newcommand{\ra}{\rangle}

\def\equ#1{(\ref{#1})}

\newcommand{\cost}{\theta}
\newcommand{\duals}{\nu}
\newcommand{\nb}{\SN}

\def\##1{\relax\ifmmode\mathchoice      
{\mbox{\boldmath$\displaystyle#1$}}
{\mbox{\boldmath$\textstyle#1$}}
{\mbox{\boldmath$\scriptstyle#1$}}
{\mbox{\boldmath$\scriptscriptstyle#1$}}\else
\hbox{\boldmath$\textstyle#1$}\fi}
\newcommand{\SG}{{\cal G}}

\newcommand{\SE}{{\cal E}}

\newcommand{\SX}{{\cal X}}

\newcommand{\SY}{{\cal Y}}

\newcommand{\SN}{{\cal N}}

\newcommand{\SP}{{\cal P}}

\newcommand{\SL}{{\cal L}}

\newcommand{\SV}{{\cal V}}
\newcommand{\SU}{{\cal U}}
\newcommand{\BR}{\mathbb R}

\newcommand{\BM}{\mathbb M}
\newcommand{\BD}{\mathbb D}

\definecolor{checkColor}{RGB}{205,173,0}
\definecolor{fettrot}{RGB}{255,10,10}

\newcommand{\dualspace}{\ensuremath \BR^{\BD}}
\newcommand{\primalspace}{\ensuremath \BR_+^{\BM}}
\newcommand{\eg}{e.g.}

\title{Getting Feasible Variable Estimates From Infeasible Ones: MRF Local Polytope Study}
\author{Bogdan Savchynskyy, Stefan Schmidt\\
       bogdan.savchynskyy@iwr.uni-heidelberg.de,
       schmidt@math.uni-heidelberg.de \\
       Heidelberg Collaboratory for Image Processing\\
       University of Heidelberg\\
       Speyerer Str. 6\\
       69115 Heidelberg, Germany
}

\begin{document}

\maketitle

\begin{abstract}
This paper proposes a method for construction of approximate {\em feasible} primal solutions from dual ones for large-scale optimization problems possessing certain separability properties. Whereas infeasible primal estimates can typically be produced from (sub-)gradients of the dual function, it is often not easy to project them to the primal feasible set, since the projection itself has a complexity comparable to the complexity of the initial problem. We propose an alternative efficient method to obtain feasibility and show that its properties influencing the convergence to the optimum are similar to the properties of the Euclidean projection.
We apply our method to the local polytope relaxation of inference problems for Markov Random Fields and demonstrate its superiority over existing methods.
\end{abstract}

\section{Introduction}

Convex relaxations of combinatorial problems appearing in computer vision, processing of medical data, or analysis of transport networks often contain millions of variables and hundreds of thousands of constraints. 
It is also quite common to employ their dual formulations to allow for more efficient optimization, which due to strong duality 
delivers also primal solutions. Indeed, approximate primal solutions can usually be reconstructed from (sub-)gradients of the dual objective. However, these are typically infeasible. Because of the problem size, only first order methods (based on the function and its (sub-)gradient evaluation only) can be applied. Since feasibility is not guaranteed up to the optimum, it is hardly attainable for such methods because of their slow convergence.  The classical trick --- projection to the feasible set --- can not be used efficiently because of the problem size. 

A striking example of such a situation, which we explore in the paper, is the reconstruction of feasible primal estimates for local polytope relaxations of Markov random field (MRF) inference problems  \citep{Schlesinger76,Werner07,WainwrightBook08}. 

\paragraph{Motivation: Why Feasible Primal Estimates Are Needed.} It is often the case for convex relaxations of combinatorial problems that not a relaxed solution, but an integer approximation thereof is used in applications. 
Such integer primal solutions can be obtained from the dual ones due to the complementary slackness condition and using heuristic local search procedures \citep{Werner07,Kolmogorov06,Ravikumar10}. 
However, a sequence of {\em feasible} solution estimates of the {\em relaxed problem} converging to the optimum guarantees vanishing of the corresponding duality gap, and hence (i) determines a theoretically sound stopping condition \citep{Boyd2004}; (ii) provides a basis for the comparison of different optimization schemes for a given problem; (iii) allows for the construction of adaptive optimization schemes depending on the duality gap, for example adaptive step-size selection in subgradient-based schemes \citep{Komodakis10journal,Kappes2012} or adaptive smoothing selection procedures for non-smooth problems \citep{SavchynskyyUAI2012}. Another example is the tightening of relaxations with cutting-plane based approaches \citep{SontagEtAl_uai08}.

\paragraph{Contribution.} 
We propose an efficient and well-scalable method for constructing feasible points from infeasible ones for a certain class of separable convex problems. The method guarantees convergence of the constructed feasible point sequence to the optimum of the problem if only this convergence holds for their infeasible counterparts. We theoretically and empirically show how this method works in a local polytope relaxation framework for MRF inference problems.
We formulate and prove our results in a general way, which allows to apply them to arbitrary convex optimization problems having a similar separable structure. 

\paragraph{Formulation of the Main Result.}
We start by stating the main result of the paper for a separable linear programming problem. The result has a special form, which appears in the MRF energy minimization problem. This example illustrates the idea of the method and avoids shading it with numerous technical details. We refer to Sections~\ref{OptimizingProjection_Section} and~\ref{MRFInference_section} for all proofs, special cases and generalizations.

Let $\lan\cdot,\cdot \ran$ denote an inner product of two vectors in a Euclidean space. Let $\BR_{+}^n$ denote the non-negative cone of the $n$-dimensional Euclidean space~$\BR^n$. Let $I=\{1,\dots,N\}$, $J=\{1,\dots,M\}$, be sets of integer indexes and $\nb(j),\ j\in J$, be a collection of subsets of $I$. Let further $x\in\BR_+^{nI}$ be a collection of $(x_i\in \BR_+^n,\ i\in I)$ and $y\in\BR_+^{mJ}$ denote $(y_j\in \BR_+^m,\ j\in J)$. Let $A_{ij},\ i\in I,\ j\in J$ and $B_i,\ i\in I$ be matrices of dimensions $m\times n$ and $n\times k$ for some $k < n$ and let $c_i\in\BR^k$. Consider the following separable linear programming problem in the standard form
\begin{align}\label{LPexample_equ}
 \min_{x\in\BR_+^{nI}\atop y\in\BR_+^{mJ}} & \sum_{i=1}^N\lan a_i,x_i\ran + \sum_{j=1}^{M}\lan b_j,y_j\ran\\
&A_{ij} y_j=x_i,\ i\in \nb(j),\ j\in J\,,\nonumber\\
&B_i x_i=c_i,\ i\in I\,. \nonumber
\end{align}

Let $D$ be the feasible set of the problem~\equ{LPexample_equ} and the mapping $\SP\colon \BR_+^{nI}\times \BR_+^{mJ} \to D$ be defined such that $\SP(x,y)=(x',y')$, where
\begin{align}\label{optimizingProjectionExample_equ}
 & x_i',\ i\in I\ \mbox{are Euclidean projections of}\ x_i\ \mbox{to the sets}\ \{\tilde x_i\in\BR_+^n\colon B_i\tilde  x_i=c_i\}\,;\\
 & y_j':= \arg \min\limits_{y_j\in\BR_+^{m}} \lan b_j,y_j\ran\ \mbox{s.t.}\ A_{ij} y_j=x_i',\ i\in \nb(j)\,.\label{optimizingProjectionExample_equ2}
\end{align}

{\bf The  main result} of this paper states that {\em from the convergence of $(x^t,y^t)\in\BR^{nI}\times\BR^{mJ}$, $t=1,2,\dots\infty$ to the set of optimal solutions of~\equ{LPexample_equ} it follows that $\SP(x^t,y^t)$ converges to the set of optimal solutions as well.}

Please note that 
\begin{itemize}
 \item $\SP(x^t,y^t)$ is always feasible due to its construction;
 \item contrary to the Euclidean projection onto the set $D$, to compute $\SP(x^t,y^t)$ one has to solve many, but small quadratic and linear optimization problems~\equ{optimizingProjectionExample_equ}-\equ{optimizingProjectionExample_equ2}, assuming that $n\ll I$, $m\ll J$ and $N(J) \ll I$. To this end such powerful, but not very well scalable tools as simplex or interior point methods can be used due to the small size of these problems.
\end{itemize}

In Section~\ref{OptimizingProjection_Section} we additionally show how the convergence speed of $\SP(x^t,y^t)$ depends on coefficients $a_i$ and $b_i$.

Assuming that the set $D$ corresponds to the local polytope, variables $x_i$ and $y_i$ to unary and binary "max-marginals" and weights $a_i$ and $b_j$ to unary and pairwise potentials respectively, this result allows for an efficient estimation of feasible primal points from infeasible ones for MRF energy minimization algorithms, which has been considered as a non-trivial problem in the past \citep{Werner07}.

\paragraph{Related Work on MRF Inference}
The two most important inference problems for MRF's are maximum a posteriori (MAP) inference and marginalization \citep{WainwrightBook08}. Both are intractable in general and thus both require some relaxation. The simplest convex relaxation for both is based on exchanging an underlying convex hull of the feasible set, the marginal polytope, by an approximation called the local polytope \citep{WainwrightBook08}. 
However, even with this approximation the problems remain non-trivial, though solvable, at least theoretically. A series of algorithmic schemes were proposed to this end  for the local polytope relaxations of both MAP \citep{Komodakis10journal,SchlGig_12_usim2007,Ravikumar10,Savchynskyy11,SchmidtEMMCVPR11,Kappes2012,SavchynskyyUAI2012,MeshiGloversonECML11,MartinsICML11} and marginalization \citep{WainwrightLogPartition2005,Jancsary2011d,HazanShashua,HazanUAI2012}. It turns out that the corresponding dual problems have dramatically less variables and contain very simple constraints \citep{Werner07,WernerReport09}, hence they can even be formulated as unconstrained problems as it is done by~\cite{SchlGig_12_usim2007} and \cite{Kappes2012}. Therefore, most of the approaches address optimization of the dual objectives. A common difficulty for such approaches is the computation of a {\em feasible} relaxed primal estimate from the current dual one. {\em Infeasible} estimates can typically be obtained from the subgradients of the dual function as shown by~\cite{Komodakis10journal} or from the gradients of the smoothed dual as done by~\cite{Johnson07}, \cite{WernerReport09}, and \cite{Savchynskyy11}.

Even some approaches working in the primal domain \citep{HazanShashua,MartinsICML11,SchmidtEMMCVPR11,MeshiGloversonECML11} maintain infeasible primal estimates, whilst feasibility is guaranteed only in the limit.

Quite efficient primal schemes based on graph cuts proposed by~\cite{BoykovFastApproxGraphCuts2001} do not solve the problem in general and optimality guarantees provided by them are typically too weak. Hence we do discuss neither these here, nor the widespread message passing and belief propagation \citep{Kolmogorov06,WeissF01} methods, which also do not guarantee the attainment of the optimum of the relaxed problem.

\paragraph{Forcing Feasibility of Primal Estimates}

The literature on obtaining feasible primal solutions for MRF inference problems from infeasible ones is not very vast. Apart from our conference papers \citep{Savchynskyy11,SchmidtEMMCVPR11,SavchynskyyUAI2012}) preceding this work, we are aware of only two recent works contributing to this topic, by~\cite{SchlesingerEMMCVPR11} and~\cite{WernerCSM2011}. 

The method proposed by~\cite{SchlesingerEMMCVPR11} is formulated in the form of an algorithm able to determine whether a given solution accuracy $\eps$ is attained or not. To this end it restricts the set of possible primal candidate solutions and solves an auxiliary quadratic programming (QP) problem. However, this approach is unsuited to compute {\em the actually attained} $\eps$ directly and the auxiliary QP in the worst case grows linearly with the size of the initial linear programming problem. Hence obtaining a feasible primal solution becomes prohibitively slow as the size of the problem gets larger.

Another closely related method was proposed by~\cite{WernerCSM2011}. It is, however, only suited to determine whether a given solution of the dual problem is an optimal one. This makes it non-practical, since the state-of-the-art methods achieve the exact solution of the considered problem only in the limit, after a potentially infinite number of iterations.


\paragraph{Content and Organization of the Paper}

Besides this introduction the paper contains five further sections. In Section~\ref{OptimizingProjection_Section} we describe a general formulation and mathematical properties of the {\em optimizing projection} $\SP(x,y)$, as already introduced for a special case in~\equ{optimizingProjectionExample_equ}-\equ{optimizingProjectionExample_equ2}. We do this without relating it to inference in MRFs, to allow readers not familiar with the latter to catch the idea. Section~\ref{MRFInference_section} is devoted to both MAP and marginalization inference problems for MRF's and specifies how the optimizing projection can be constructed for corresponding primal and dual problems. In Section~\ref{applications_sec} we provide a list of algorithmic schemes working in the dual domain and show how primal estimates can be reconstructed from dual ones for all of them.  The feasibility of the estimates is guaranteed by our optimizing projection method.
The last Sections~\ref{demo_Section} and~\ref{Conclusions_Section} contain the experimental evaluation and conclusions, respectively.

\section{Optimizing Projection}\label{OptimizingProjection_Section}

Let us denote by $\Pi_{D}\colon\BR^n\to D$ an Euclidean projection to a set $D\subset\BR^n$. 
Let $X\subseteq\BR^n$ and $Y\subseteq\BR^m$ be two subsets of Euclidean spaces and $D\subset X\times Y$ be a closed convex set. We will denote as $D_{\SX}$ the set $\{x\in\SX\ | \exists y\in\SY\colon (x,y)\in D\}$, that is the projection of $D$ to $\SX$.

The main definition of the paper introduces the notion of the {\em optimizing projection} in its general form. A possible simplification and the corresponding discussion follow the definition.

\begin{definition}\label{optimizingProjection_def}
 Let $f\colon X\times Y\to\BR$ be a continuous convex function of two variables.  
 The mapping $\SP_{f,D}\colon X\times Y \to D$ such that $\SP_{f,D}(x,y)=(x',y')$ defined as
\begin{align}
 x'& = \Pi_{D_{\SX}}(x)\,,\label{euclideanProjectionStep_def}\\
 y' & = \min_{y\colon (x',y)\in D} f(x',y)\,,\label{partialMin_def}
\end{align}
is called an {\em optimizing projection} onto the set $D$ w.r.t.\ the function $f$.
\end{definition}

The definition shows the way to get a feasible point $(x',y')\in D$ from an arbitrary infeasible one $(x,y)$. Of course, getting just any feasible point is not a big issue in many cases. However, as we will see soon, the introduced optimizing projection possesses properties similar to the properties of a standard Euclidean projection, which makes it a useful tool in cases when its computation is easier than the one needed for the Euclidean projection. To this end both the partial projection~\equ{euclideanProjectionStep_def} and the partial minimization~\equ{partialMin_def} should be efficiently computable.

The role of projection~\equ{euclideanProjectionStep_def} is to make $x$ ``feasible'', i.e.\ to guarantee for $x'$ that there is at least one $y\in\SY$ such that $(x',y)\in D$, which guarantees the definition to be well-defined. If this condition holds already for $x$, it is easy to see that $x'=x$ and hence computing~\equ{euclideanProjectionStep_def} is trivial. We will call such $x$ {\em feasible} w.r.t.\ $D$. Indeed, in~\equ{euclideanProjectionStep_def} one can apply an arbitrary projection, since they all satisfy the mentioned property. However, we provide our analysis for Euclidean projections only.

\begin{example}
 Consider the linear programming problem~\equ{LPexample_equ} from the introduction. It is reasonable to construct an optimizing projection $\SP_{f,D}(x,y)$ for it as in~\equ{optimizingProjectionExample_equ}-\equ{optimizingProjectionExample_equ2}, denoting with $f$ and $D$ the objective function and the feasible set of the problem~\equ{LPexample_equ}.
\end{example}

We will deal with objective functions, which fulfill the following definition:
\begin{definition}\label{LipschitzContinuity_def}
 A function $f\colon X\times Y\to\BR$ is called Lipschitz-continuous w.r.t.\ its first argument $x$, if there exists a finite constant $L_X(f)$, such that
 $\forall y\in Y,\ x,x'\in X$ 
\begin{equation}
\vert f(x,y)-f(x',y)\vert \le L_X(f)\Vert x-x'\Vert 
\end{equation}
\end{definition}
holds. 
Similarly $f$ is Lipschitz-continuous w.r.t.\ 
\begin{itemize}
 \item $y$ if $\vert f(x,y)-f(x,y')\vert \le L_Y(f)\Vert y-y'\Vert$ for all $x\in X,\ y,y'\in Y$ and  some constant $L_Y(f)$;
 \item $z=(x,y)$ if $\vert f(x,y)-f(x',y')\vert \le L_{XY}(f)\Vert z-z'\Vert$ for all $z,z'\in X\times Y$ and  some constant 
$L_{XY}(f)$\,.
\end{itemize}

The following theorem specifies the main property of the optimizing projection, namely its continuity with respect to the optimal value of the function $f$.
\begin{theorem}\label{mainPropertyOptProjection_tm}
 Let $f$ be convex and Lipschitz-continuous w.r.t.\ its arguments $x$ and $y$ and let $f^*$ be the minimum of $f$ on the set $D$. Then for all $z=(x,y)\in X\times Y$ 
\begin{equation}\label{mainPropertyOptProjection_equ}
 \vert f(\SP_{f,D}(x,y)) - f^*\vert \le \vert f(x,y) - f^*\vert + (L_X(f) + L_Y(f))\Vert z-\Pi_{D}(z)\Vert\,
\end{equation}
holds.
If additionally $x$ is feasible w.r.t.\ $D$ the tighter inequality holds:
\begin{equation}\label{mainPropertyOptProjectionTrivialX_equ}
 \vert f(\SP_{f,D}(x,y)) - f^*\vert \le \vert f(x,y) - f^*\vert + L_Y(f)\Vert z-\Pi_{D}(z)\Vert\,.
\end{equation}
\end{theorem}

\begin{proof}
We will denote $(x^p,y^p)=z^p=\Pi_{D}(z)$ and $(x',y')=\SP_{f,D}(x,y)$.
Note that 
\begin{itemize}
 \item from $f^*\leq f(x',y') \leq f(x',y'')$ for any $y''\in Y$ such that $(x',y'')\in D$ it
follows that
\begin{equation}\label{optYpt_equ}
f^*\leq f(x',y') \leq f(x',y^p)\,,
\end{equation}
\item from $\Vert z-z^p \Vert = \sqrt{\Vert x-x^p \Vert^2 + \Vert y -y^p \Vert^2}$ it follows that
\begin{equation}\label{distZpt_equ}
 \Vert y -y^p \Vert \leq  \Vert z-z^p \Vert\ \mbox{and}\ \Vert x -x^p \Vert \leq  \Vert z-z^p \Vert\,.
\end{equation}
\item according to~\equ{euclideanProjectionStep_def} 
$x'=\Pi_{D_{\SX}}(x)=\arg\min_{\tilde x\in D_{\SX}}\Vert x-\tilde x\Vert$ and hence $\Vert x-x'\Vert \le \Vert x-x^p\Vert$ 
since $x^p\in D_{\SX}$. Combining this with~\equ{distZpt_equ} we obtain
\begin{equation}\label{distZpvsXp_equ}
 \Vert x -x' \Vert \le \Vert z-z^p \Vert\,.
\end{equation}
 
\end{itemize}

The proof follows from the following sequence of inequalities: 
\begin{multline}\label{projectionConvergence_inequ}
 \vert f(\SP_{f,D}(x,y))-f^*\vert =\vert f(x',y') - f^*\vert \stackrel{\equ{optYpt_equ}}{\leq} \vert f(x',y^p) - f^*\vert \\
\le \vert f(x',y^p) - f(x',y)\vert + \vert f(x',y) - f^*\vert \le L_Y(f)\Vert y-y^p\Vert +\vert f(x',y) - f^*\vert\\ 
\stackrel{\equ{distZpt_equ}}{\leq} L_Y(f)\Vert z-z^p\Vert +\vert f(x',y) - f^*\vert \,. 
\end{multline}

Estimate~\equ{mainPropertyOptProjectionTrivialX_equ} follows from~\equ{projectionConvergence_inequ} assuming that $x'=x$.

The proof for the more general case~\equ{mainPropertyOptProjection_equ} follows from~\equ{projectionConvergence_inequ} and Lipschitz-continuity of $f$ w.r.t.\ $x$:
\begin{multline}
 \vert f(\SP_{f,D}(x,y))-f^*\vert\stackrel{\equ{projectionConvergence_inequ}}{\leq} L_Y(f)\Vert z-z^p\Vert +\vert f(x',y) - f^*\vert \\
\le L_Y(f)\Vert z-z^p\Vert +  \vert f(x',y) - f(x,y)\vert + \vert f(x,y) - f^*\vert \\
\le L_Y(f)\Vert z-z^p\Vert + L_X(f) \Vert x'-x\Vert + \vert f(x,y) - f^*\vert \\
\stackrel{\equ{distZpvsXp_equ}}{\le} L_Y(f)\Vert z-z^p\Vert + L_X(f) \Vert z-z^p\Vert + \vert f(x,y) - f^*\vert
= (L_Y(f) +  L_X(f))\Vert z-z^p\Vert  + \vert f(x,y) - f^*\vert\,.
\end{multline}
\end{proof}

Theorem~\ref{mainPropertyOptProjection_tm} basically states that if the sequence $z^t=(x^t,y^t)\in X\times Y,\ t=1,\dots,\infty$ weakly converges to the optimum of $f$, then the same holds also for $\SP_{f,D}(x^t,y^t)$. Moreover, the rate of convergence is preserved up to a multiplicative constant.  Please note that $\SP_{f,D}(x,y)$ actually does not depend on $y$, it is needed only for the convergence estimates~\equ{mainPropertyOptProjection_equ} and~\equ{mainPropertyOptProjectionTrivialX_equ}, but not for the optimizing projection itself.

Let us provide an analogous bound for the Euclidean projection to get an idea how good the estimate given by Theorem~\ref{mainPropertyOptProjection_tm} is. Let $z$ and $z^p$ be defined as in the proof of the theorem. Then
 \begin{align}\label{EuclideanProjectionBound_equ}
  \vert f(z^p)-f^* \vert \le \vert f(z^p)-f(z)\vert + \vert f(z)-f^*\vert \le  \vert f(z)-f^*\vert +  L_{XY}(f)\Vert z-z^p\Vert\,.
 \end{align}

We see that bounds~\equ{mainPropertyOptProjection_equ} and~\equ{EuclideanProjectionBound_equ} for the optimizing mapping and Euclidean projection differ only by a constant factor: in the optimizing mapping, the Lipschitz continuity of the objective $f$ is considered w.r.t.\ to each variable $x$ and $y$ separately, whereas the Euclidean projection is based on the Lipschitz continuity w.r.t.\ the pair of variables $(x,y)$.

The following lemma shows the difference between these two Lipschitz constants. Together with the next one it will be intensively used in the rest of the paper: 
\begin{lemma}\label{linearFunctionLipschitzConst_lem}
 The linear function $f(x,y)=\lan a, x\ran+\lan b, y\ran$ is Lipschitz-continuous with Lipschitz constants $L_X(f)\le\Vert a\Vert$, $L_Y(f)\le\Vert b \Vert$ and $L_{XY}(f)\le\sqrt{L_X(f)^2+L_Y(f)^2}$.
\end{lemma}
\begin{proof}
All three Lipschitz-constants are derived from the Cauchy-Bunyakovsky-Schwarz inequality
\begin{equation}
\lan c,\nu\ran\le \Vert c\Vert\cdot \Vert \nu\Vert,\ c,\nu\in \BR^N\,
\end{equation}
applied respectively to $x$, $y$ and $z=(x,y)$ in place of $\nu$.
\end{proof}

\begin{lemma}\label{entropyFunctionLipschitzConst_lem}
 The function $f(z)=\lan a,z\ran + \sum_{i=1}^N z_i\log z_i$, where $\log$ denotes the natural logarithm, is Lipschitz-continuous in the box $[\eps,M]^N\ni z,\ \eps>0, M>\eps$ with Lipschitz-constant 
\begin{equation}\label{entropyLipschitzConst_equ}
L_{XY}(f)\le\Vert a\Vert + N \max\{\vert 1+\log\eps\vert, \vert 1+\log M\vert\}\,. 
\end{equation}
\end{lemma}
\begin{proof}
The function $f_i(z_i)=z_i\log z_i$ of a single variable is differentiable on $[\eps,M]$ and its derivative $f'_i(z_i)=1+\log z_i$ is monotone increasing, hence $f_i(z_i)$ is convex. This implies $f_i(z_i)-f_i(z'_i)\le f'_i(z_i) (z_i-z'_i)$ and $\vert f_i(z_i)-f_i(z'_i)\vert \le \vert f'_i(z_i)\vert \vert (z_i-z'_i) \vert$. Taking into account that due to monotonicity $\vert f'_i(z_i) \vert \le \max\{\vert 1+\log\eps \vert, \vert 1+\log M\vert\}$ for $z_i\in [\eps,M]$, and using the fact that $L(f_1+f_2)\le L(f_1)+L_f(f_2)$ together with Lemma~\ref{linearFunctionLipschitzConst_lem}, one obtains~\equ{entropyLipschitzConst_equ}. 
\end{proof}

\section{MRF Inference and Optimizing Projections}\label{MRFInference_section}

In this section we consider optimization problems related to inference in MRF's and construct corresponding optimizing projections. We switch from the general mathematical notation used in the previous sections to the one specific for the considered field, in particular we mostly follow the book of~\cite{WainwrightBook08}. 

The section consists of two subsections. 
The first one describes the MAP-inference problem for MRFs, its local polytope relaxation and the primal and dual formulations. In that subsection we show how the optimizing projection introduced in Section~\ref{OptimizingProjection_Section} can be applied to obtain both primal and dual feasible estimates.


The second part is devoted to a decomposition-based dual formulation, and it introduces basic notions for Section~\ref{applications_sec}. Additionally we show here how feasible primal estimates can be obtained for the tree-reweighted free energy introduced by~\cite{WainwrightLogPartition2005}.

\subsection{Local Polytope Relaxation}

This section is devoted to the maximum-a-posteriori (MAP) inference problem for Markov random fields, known also as an energy minimization problem. We derive primal and dual formulations for the so-called local polytope relaxation of the problem, analyze their separability properties and construct the corresponding optimizing projections.  

\subsubsection{Primal Problem}

\paragraph{Preliminaries}
Let $\SG=(\SV,\SE)$ be an undirected graph, where $\SV$ is a finite set of nodes and $\SE\subset\SV\times\SV$ is a set of edges. Let further $\SX_v,\ v\in\SV$, be finite {\em sets of labels}. The set $\SX=\otimes_{v\in\SV}\SX_v$, where $\otimes$ denotes the Cartesian product, will be called {\em labeling set} and its elements $x\in\SX$ are {\em labelings}. Thus each labeling is a collection $(x_v\colon v\in\SV)$ of labels. To shorten notation we will use $x_{uv}$ for a pair of labels $(x_u,x_v)$ and $\SX_{uv}$ for $\SX_u\times\SX_v$. 
The collections of numbers $\theta_{v,x_v},\ v\in\SV,\ x_v\in\SX_v$ and $\theta_{uv,x_{uv}},\ uv\in\SE,\ x_{uv}\in\SX_{uv}$ will be called {\em unary} and {\em pairwise potentials}, respectively.
The collection of all potentials will be denoted by $\theta$. 

The problem is to compute the labeling $x$ which minimizes the energy function $E_{\SG}$:  
\begin{equation}\label{minEnergy_equ2}
\min_{x\in\SX}E_{\SG}(\theta,x)=\min_{x\in\SX}\left\{\sum_{v\in\SV}\theta_{v,x_v}+\sum_{uv\in\SE}\theta_{uv,x_{uv}}\right\}.
\end{equation}

An alternative way of writing problem~\equ{minEnergy_equ2} is to express it in the form of a scalar product of the vector $\theta$, 
 denoting the collection of all $\theta_{v,x_v},\ v\in\SV,\ x_v\in\SX_v$ and $\theta_{uv,x_{uv}},\ uv\in\SE,\ x_{uv}\in\SX_{uv}$, with a suitably constructed binary vector $\phi(x),\ x\in\SX$: 
\begin{equation}\label{decreteObjective_equ}
\min_{x\in\SX}\lan \theta,\phi(x)\ran\, .   
\end{equation}

Denoting $\BR^{\sum_{v\in\SV}|\SX_v|+\sum_{uv\in\SE}|\SX_{uv}|}$ as $\BR(\BM)$ and the corresponding non-negative cone \\
$\BR^{\sum_{v\in\SV}|\SX_v|+\sum_{uv\in\SE}|\SX_{uv}|}_{+}$ as  $\primalspace$, we relax~\equ{minEnergy_equ2} to the linear programming problem \citep{Schlesinger76,Werner07}
 \begin{align}\label{LPRelaxation_equ}
&\min_{\mu\in\primalspace}\sum_{v\in\SV}\sum_{x_v\in\SX_v}\theta_{v,x_v}\mu_{v,x_v} 
+\sum_{uv\in\SE}\sum_{x_{uv}\in\SX_{uv}}\theta_{uv,x_{uv}}\mu_{uv, x_{uv}}\nonumber\\ 
&\mbox{s.t.}
\begin{array}{l}
 \sum_{x_v\in\SX_v}\mu_{v,x_v}=1,\ v\in\SV\,,\\
 \sum_{x_v\in\SX_v}\mu_{uv, x_{uv}}=\mu_{u,x_u},\ x_u\in\SX_u,\ uv\in\SE\,,\\
 \sum_{x_u\in\SX_u}\mu_{uv, x_{uv}}=\mu_{v,x_v},\ x_v\in\SX_v,\ uv\in\SE\,.\\
\end{array}
\end{align}
The constraints in~\equ{LPRelaxation_equ} form the {\em local polytope}, later on denoted as $\SL$. 
Slightly abusing notation, we will briefly write problem~\equ{LPRelaxation_equ} as 
\begin{equation}\label{localPolytopeEnergy_equ}
 \min_{\mu\in\SL}E(\mu):=\min_{\mu\in\SL}\lan\theta,\mu\ran\,.
\end{equation}

\begin{remark}
Please note that introducing additional constraints
\begin{equation}\label{integralityConstraint:equ}
\mu_{v,x_v}\in\{0,1\} \mbox{ and } \mu_{uv,x_{uv}}\in\{0,1\},\  v\in\SV,\ uv\in\SE,\ x_v\in{\SX_v},\ x_{uv}\in{\SX_{uv}}\,,
\end{equation}
would make~\equ{LPRelaxation_equ} equivalent to~\equ{minEnergy_equ2}. 
Each labeling $x\in\SX$ corresponds to some point $\mu$ satisfying the conditions of~\equ{LPRelaxation_equ} and~\equ{integralityConstraint:equ}, namely that having $\mu_{v,x'_v}=1$ iff $x'_v=x_v$, 0 else.
\end{remark}

\paragraph{Optimizing Projection}
 We will denote as $\theta_w$ and $\mu_w,\ w\in\SV\bigcup\SE$, the collections of $\theta_{w,x_w}$ and $\mu_{w,x_w},\ x_w\in\SX_w$, respectively. Hence the vectors $\theta$ and $\mu$ become collections of $\theta_w$ and $\mu_w,\ w\in\SV\bigcup\SE$. The $n$-dimensional simplex $\{x\in\BR_+^n\colon \sum_{i=1}^n x_i=1\}$ will be denoted as $\Delta(n)$.
 
Problem~\equ{LPRelaxation_equ} has a separable structure similar to~\equ{LPexample_equ}, i.e.\ for suitably selected matrices $A_{uv}$ it can be written as
 \begin{align}\label{LPRelaxation_equ_separable}
&\min_{\mu\in\BR(\BM)}\lan \theta_v,\mu_v\ran + \lan\theta_{uv},\mu_{uv}\ran\nonumber\\ 
&\mbox{s.t.}
\begin{array}{ll}
 \mu_v\in\Delta(|\SX_v|), & v\in\SV\,,\\
 A_{uv}\mu_{uv}=\mu_v,\ \mu_{uv}\ge 0, & uv\in\SE\,.
\end{array}
\end{align}
Note that under fixed $\mu_v$, the optimization of~\equ{LPRelaxation_equ_separable} splits into small independent subproblems, one for each $uv\in\SE$. We will use this fact to compute the optimizing projection onto the local polytope $\SL$ as follows.

Let $\mu_{\SV}$ and $\mu_{\SE}$ be collections of primal variables corresponding to graph nodes and edges respectively, i.e.\ $\mu_{\SV}=(\mu_v,\ v\in\SV)$ and $\mu_{\SE}=(\mu_{uv},\ uv\in\SE)$. The corresponding subspaces will be denoted by $\BR(\BM_{\SV})$ and $\BR(\BM_{\SE})$.
Then according to~\equ{LPRelaxation_equ_separable} and Definition~\ref{optimizingProjection_def}, the optimizing projection $\SP_{E,\SL}\colon \BR(\BM_{\SV})\times\BR(\BM_{\SE})\to\SL$ maps $(\mu_{\SV},\mu_{\SE})$ to $(\mu'_{\SV},\mu'_{\SE})$ defined as
\begin{align}
 &\ \mu'_v\ \ =\Pi_{\Delta(|\SX_v|)}(\mu_v),\ v\in\SV\label{primalOptMappingSimplexProjection_equ}\,,\\
 &\begin{array}{ll}
 \mu'_{uv} = \arg &\min\limits_{\mu_{uv}\ge 0}\lan \theta_{uv},\mu_{uv}\ran\\
	   & \mbox{s.t.\ }  A_{uv}\mu_{uv}=\mu'_v
\end{array},\ uv\in\SE\,. \label{primalOptMappingTransportpProblem_equ}
\end{align}

Note that both~\equ{primalOptMappingSimplexProjection_equ} and~\equ{primalOptMappingTransportpProblem_equ} can be computed very efficiently. Projection to a simplex in~\equ{primalOptMappingSimplexProjection_equ} can be done \eg\ by method of~\cite{Michelot86}. The optimization problem in~\equ{primalOptMappingTransportpProblem_equ} constitutes a small-sized {\em transportation problem} well-studied in linear programming, see for example the text-book of~\cite{BazaraaJarvis77}.

Let us apply Theorem~\ref{mainPropertyOptProjection_tm} and Lemma~\ref{linearFunctionLipschitzConst_lem} to the  optimizing projection~$\SP_{E,\SL}$ introduced in Definition~\ref{optimizingProjection_def}. According to these, the convergence rate of a given sequence $\mu^t\in\BR(\BM)$ in the worst case slows down by a factor $L_{\BM_{\SV}}(E) + L_{\BM_{\SE}}(E)\le\Vert \theta_{\SV}\Vert + \Vert \theta_{\SE}\Vert$.
This factor can be quite large, but since the optimum $E^*$ grows together with the value $\Vert\theta_{\SV}\Vert + \Vert \theta_{\SE}\Vert$, its influence on the obtained {\em relative} accuracy is typically much less than the value itself. 

\begin{remark}
However, if $\theta$ contains "infinite" numbers, typically assigned to pairwise factors $\theta_{\SE}$ to model "hard" constraints, both optimizing and Euclidean projections can be quite bad, which is demonstrated by the following example, depicted in Fig.~\ref{fig:exp0}:
$\SV=\{v,u\}$, $\SE=uv$, $\SX_v=\SX_u=\{0,1\}$, $\theta_{00}=\theta_{11}=\theta_{01}=0$, $\theta_{10}=\infty$. If now $\mu_{v,1}>\mu_{u,1}$, optimizing w.r.t.\ $\mu_{uv}$ leads to $\theta_{10}\cdot\mu_{vu,10}=\infty\cdot(\mu_{v,1}-\mu_{u,1})$, whose value can be arbitrary large, depending on the actual numerical value approximating $\infty$. And since neither the optimizing projection nor the Euclidean one take into account the actual values of pairwise factors when assigning values to $\mu_{\SV}$, the relation $\mu_{v,1}>\mu_{u,1}$ is not controlled.
\end{remark}

We provide an additional numerical simulation related to infinite values of pairwise potentials in Section~\ref{demo_Section}.

\begin{figure}[t]
\begin{center}
\mbox{\includegraphics[width=0.30\linewidth]{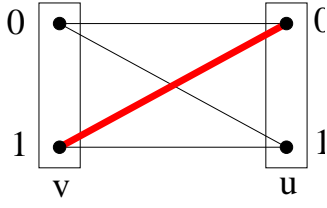}}
\end{center}

\caption{ A pairwise factor of a graphical model. Vertically oriented rectangles correspond to graph nodes $v$ and $u$, black circles inside to variable states $0$ and $1$. Lines connecting states in two nodes correspond to different values of pairwise potentials. Potentials corresponding to all states and pairs except the one $\theta_{vu,10}$ denoted by a thick red line are assumed to be equal $0$, whereas $\theta_{vu,10}$ is assumed to be infinitely large. Clearly optimal values of primal variables $\mu_v$ and $\mu_u$ assigned to these nodes always satisfy $\mu_{v,1} \le \mu_{u,1}$. Otherwise, due to local polytope constraints, an optimal value $\mu_{vu,10}$ corresponding to the infinite pairwise factor value is equal to $\mu_{v,1} - \mu_{u,1}$. This corresponds to arbitrary large primal objective values  even for very small positive values of $\mu_{v,1} - \mu_{u,1}$.
}
\label{fig:exp0}
\end{figure}

\begin{remark}[Higher order models and relaxations]
The generalization of the optimizing projection~\equ{primalOptMappingSimplexProjection_equ}-\equ{primalOptMappingTransportpProblem_equ} for both higher order models, and higher order local polytopes introduced by~\citet[Sec.~8.5]{WainwrightBook08} is quite straightforward. The underlying idea remains the same: one has to fix a subset of variables such that the resulting optimization problem splits into a number of small ones.
\end{remark}

\begin{remark}[Efficent representation of the relaxed primal solution]
Note that since the pairwise primal variables $\mu_{\SE}$ can be easily recomputed from unary ones $\mu_{\SV}$, it is sufficient to store only the latter if one is not interested in specific values of pairwise variables~$\mu_{\SE}$. Because of possible degeneracy, there may exist more than a single vector $\mu_{\SE}$ optimizing the energy $E$ for given $\mu_{\SV}$.
\end{remark}

\subsubsection{Lagrange Dual Problem}

\paragraph{Preliminaries} 

Problem~\equ{LPRelaxation_equ} can be written in a more compact form with a suitably selected matrix $A$ and vector $b$:
\begin{equation}\label{PrimalLP:equ}
 \min_{\mu\in\primalspace} \quad \la \cost,\mu\ra \quad    \text{s.t.} \quad  A\mu = b 
\end{equation}
 
Introducing the space $\dualspace:=\BR^{ |\SV| + |\SE|+\sum_{uv\in\SE}(|\SX_v|+|\SX_{u}|)}$, the dual problem reads 
  \begin{align}\label{DualLP:equ}
    \max_{\duals\in\dualspace} \quad  & \la b, \duals \ra \quad  \text{s.t.} \quad  A^\T\duals\le \cost\,.
  \end{align}

In what follows we will sometimes require an explicit form of $A^{\T}$. To this end we denote as $\nb(v)=\{u\in\SV\colon uv\in\SE\}$ the set of neighboring nodes of a node $v\in\SV$. We consider the dual variable $\duals\in\dualspace$ to consist of the following groups of coordinates: $\duals_v,\ v\in\SV$; $\duals_{uv},\ uv\in\SE$; and $\duals_{v\to u, x_v}$, $v\in\SV,\ u\in\nb(v),\ x_v\in\SX_v$. The dual~\equ{DualLP:equ} can be written explicitly \citep{Schlesinger76,Werner07} as: 

   \begin{align} 
     \max_{\duals\in\dualspace} \, & \sum_{v\in \SV}  \duals_v + \sum_{uv\in \SE}
     \duals_{uv} \label{DuaBlockLPRef:equ}  \\
\text{s.t.} \quad       
& \begin{cases}
 \cost_{v,x_v} - \sum\nolimits_{u \in
       \nb(v)}\duals_{v\to u, x_v}  \ge \duals_v\,,\ & \forall v\in \SV,\,x_v\in
     \SX_v\, , \\ 
    \cost_{uv,x_{uv}} + \duals_{u\to v, x_u} + \duals_{v\to u,x_v} \ge
   \duals_{uv},\ & \forall uv\in \SE\,,
   x_{uv}\in \SX_{uv}
   \,.
   \end{cases} \label{DuaBlockLPRefb:equ} 
   \end{align}

We will use the notation $\SU(\duals):=\lan b,\duals\ran=\sum_{v\in \SV}\duals_v + \sum_{uv\in \SE}\duals_{uv}$ for the objective function of the dual problem~\equ{DualLP:equ}.

\paragraph{Optimizing Projection} The dual~\equ{DualLP:equ} possesses clear separability as well. From~\equ{DuaBlockLPRefb:equ} it follows that after fixing all variables except $\duals_v,\ v\in\SV$, and $\duals_{uv},\ uv\in\SE$, the optimization w.r.t.\ the latter splits into a series of small and straightforward minimizations over a small set of values
\begin{align}
  \duals_v  & =\min_{x_v\in\SX_v}\cost_{v,x_v} - \sum\nolimits_{u \in \nb(v)}\duals_{v\to u, x_v},\ v\in \SV\,, \label{dualSeparableMin_v}\\ 
  \duals_{uv} & =\min_{x_{uv}\in\SX_{uv}}\cost_{uv,x_{uv}} + \duals_{u\to v, x_u} + \duals_{v\to u,x_v},\ uv\in \SE\,.\label{dualSeparableMin_uv}
\end{align}
The formula~\equ{dualSeparableMin_v} can be applied directly for each $v\in\SV$, and~\equ{dualSeparableMin_uv} accordingly for each $uv\in\SE$.

We denote by $\BD$ the dual feasible set ${\{\duals\in\BR(\BD)\colon A^{\T}\duals\le \cost\}}$. We split all dual variables into two groups. The first one will contain "messages" $\duals_{\to}=(\duals_{v\to u},\ v\in\SV,\ u\in\nb(v))$, that are variables, which reweight unary and pairwise potentials leading to improving the objective. The vector space containing all possible values of these variables will be denoted as $\BR(\BD_{\to})$. The second group will contain lower bounds on optimal reweighted unary and pairwise potentials $\duals_{0}=(\duals_w,\ w\in\SV\bigcup\SE)$. The total sum of their values constitutes the dual objective. All possible values of these variables will form the vector space $\BR(\BD_0)$.   Hence the optimizing projection $\SP_{\SU,\BD}\colon \BR(\BD_{\to})\times\BR(\BD_0)\to \BR(\BD)$ maps $(\duals_{\to},\duals_{0})$ to $(\duals'_{\to},\duals'_{0})$ as
\begin{align}
\duals'_{v\to u} & = \duals_{v\to u},\ v\in\SV,\ u\in\nb(v)\,,\label{trivialProjectionDual_equ} \\
\duals'_v  & =\min_{x_v\in\SX_v}\cost_{v,x_v} - \sum\nolimits_{u \in \nb(v)}\duals'_{v\to u, x_v},\ v\in \SV\,,\\ 
\duals'_{uv} & =\min_{x_{uv}\in\SX_{uv}}\cost_{uv,x_{uv}} + \duals_{u\to v, x_u} + \duals'_{v\to u,x_v},\ uv\in \SE\,.\label{trivialProjectionDual3_equ}
\end{align}


Equation~\equ{trivialProjectionDual_equ} corresponds to the projection~\equ{euclideanProjectionStep_def}, which has the form $\Pi_{\BR(\BD_{\to})}(\duals_{\to})= \duals_{\to0}$ and is thus trivial.

Applying Theorem~\ref{mainPropertyOptProjection_tm} and Lemma~\ref{linearFunctionLipschitzConst_lem} to the optimizing projection $\SP_{\SU,\BD}$ yields that the convergence of the projected $\duals^t$ slows down no more than by a factor $L_{\BD_0}\le\vert\sqrt{\SV}\vert + \vert \sqrt{\SE}\vert$ and does not depend on the potentials $\theta$. However, since an optimal energy value grows often proportionally to $\vert\SV\vert + \vert\SE\vert$, the influence of the factor on the estimated related precision is typically insignificant.

\subsection{Decomposition Based Dual Problem}

In this section we introduce an alternatively constructed dual objective, corresponding to the local polytope relaxation of the MAP-inference problem. We also consider a smoothed approximation thereof and show its connection to the so called tree-reweighted free energy. For the latter, we additionally construct the corresponding optimizing projection.

\subsubsection{Non-Smooth Dual Objective and Associated Subgradient}

\paragraph{Graph Decomposition}
There is an alternative way to formulate a dual problem to~\equ{minEnergy_equ2}. The corresponding technique is called Lagrangian or dual decomposition. We describe it here, because it is the base for most state-of-the-art dual algorithms for MAP and marginalization inference. In Section~\ref{applications_sec} we show how one can reconstruct primal estimates from dual ones for a range of such algorithms. 

The dual decomposition technique is based on a decomposing the graph $\SG$ into several subgraphs, which jointly cover $\SG$ and for which solving~\equ{minEnergy_equ2} is easy. The subgraphs' structures determine the underlying relaxation, and it is shown by~\cite{Komodakis10journal} that if all subgraphs are acyclic, the corresponding relaxation coincides with the local polytope relaxation, defined by~\equ{LPRelaxation_equ}. It is also known that on acyclic subgraphs problem~\equ{minEnergy_equ2} can be efficiently solved by dynamic programming.

To keep our exposition simple, we will consider the case of the graph $\SG$ being completely covered by only two acyclic subgraphs, which can be done \eg\ when $\SG$ has a grid structure. This allows to avoid technical details, preserves the main idea and can be generalized to more involved decompositions quite straightforwardly.

Let $\SG^i=(\SV^i,\SE^i),\ i=1,2$, be two \emph{acyclic} subgraphs of the \emph{master graph} 
$\SG$. Let $\SV^1=\SV^2=\SV$, $\SE^1\bigcup\SE^2=\SE$ and $\SE^1\bigcap\SE^2=\emptyset$ (\eg, $\SE^1$ may contain all horizontal edges of $\SG$ and $\SE^2$ all vertical ones if $\SG$ is a grid graph). Then the overall energy becomes the sum of the energies corresponding to these subgraphs,
\begin{equation}\label{decomposition_equ}
E_{\SG}(\theta,x) =\sum_{i=1}^{2}\sum_{v\in\SV^i}\theta^i_{v,x_v}+\sum_{uv\in\SE^i}\theta^i_{uv,x_{uv}}
                 =E_{\SG^1}(\theta^1,x)+E_{\SG^2}(\theta^2,x)\,,
\end{equation}
provided 
$\theta^i_{uv}=
\left\{
\begin{array}{ll}
\theta_{uv}, & uv\in\SE^i\\
0, & uv\notin\SE^i
\end{array}
\right.
,\ i=1,2$ 
and $\theta^1_{v,x_v}+\theta^2_{v,x_v}=\theta_{v,x_v},\ \forall v\in\SV, x_{v}\in\SX_v$. The latter condition can be represented in a parametric way as $\theta^1_{v,x_v}=\frac{\theta_{v,x_v}}{2}+\lambda_{v,x_v}$ and $\theta^2_{v,x_v}=\frac{\theta_{v,x_v}}{2}-\lambda_{v,x_v}, v\in\SV,\ x_v\in\SX_v$, where $\lambda_{v,x_v}\in\BR$. Thus we consider $\theta^i$ as a function of $\lambda=(\lambda_{v,x_v}\colon v\in\SV,\ x_v\in\SX_v)$ and have
\begin{equation}\label{upperBound_equ}
 \min_{x\in\SX}E_{\SG}(\theta,x)\ge \max_{\lambda}\sum_{i=1}^{2}\min_{x\in\SX}E_{\SG^i}(\theta^i(\lambda),x)=\min_{\mu\in\SL(\SG)}E_{\SG}(\theta,\mu).
\end{equation}
The last equation is not straightforward and we refer to the paper of~\cite{Komodakis10journal} for the proof.

\paragraph{Subgradient} 
The function 
\begin{equation}\label{DDdual_equ}
 U(\lambda):=\sum_{i=1}^{2} U^i(\lambda): =\sum_{i=1}^{2} \min_{x\in\SX}E_{\SG^i}(\theta^i(\lambda),x)=\sum_{i=1}^{2} \min_{x\in\SX}\lan\theta^i(\lambda),\phi(x)\ran
\end{equation}
is concave, but non-smooth. Its subgradient\footnote{Sometimes the term {\em supergradient} is used for concave functions.}
is equal to
\begin{equation}\label{subgradient_equ}
 \frac{\partial U}{\partial \lambda} = \sum_{i=1}^{2}  \frac{\partial U^i}{\partial \lambda} = \phi_{\SV}(x^{*1}) - \phi_{\SV}(x^{*2})\,,
\end{equation}
where $x^{*i}=\arg\min_{x\in\SX} \lan\theta^i(\lambda),\phi(x)\ran,\ i=1,2$. As we already mentioned, the $x^{*i}$ are computable by dynamic programming. This computation constitutes the basis for subgradient algorithms for MAP inference proposed by~\cite{SchlGig_12_usim2007} and~\cite{Komodakis10journal}. In Section~\ref{applications_sec} we will show how one can reconstruct primal estimates for this kind of algorithm.

\subsubsection{Smoothed Dual Objective and Associated Gradient}

Another way to maximize $U(\lambda)$ is to smooth it first and apply efficient smooth optimization techniques afterwards, as it is done \eg\ by~\cite{Savchynskyy11}. 

To obtain a smooth approximation, we replace $\min$ (or rather $-\max$) by the well-known log-sum-exp (or soft-max) function \citep{Rockafellar-Wets-04,Nesterov04}, yielding
\begin{equation}\label{smoothedUi_equ}
\hat U_{\rho}(\lambda):=\sum_{i=1}^{2}\hat U^i_{\rho}(\lambda):=-\sum_{i=1}^{2}\rho\log\sum_{x\in\SX}\exp\lan-\theta^i(\lambda)/\rho,\phi(x)\ran
\end{equation}
with \emph{smoothing parameter} $\rho$. The function $\hat U_{\rho}$ uniformly approximates $U$, as shown by \eg~\cite{Savchynskyy11}, that is,
\begin{equation}\label{smoothApprox_ineq}
 \hat U_{\rho}(\lambda)+2\rho\log|\SX|\ge U(\lambda) \ge \hat U_{\rho}(\lambda)\,.
\end{equation}

Please note that for acyclic graphs $\SG^i$ evaluating $\hat U^i_{\rho}$ (and thus $\hat U_{\rho}$) is as easy as $U^i$, and can be done by dynamic programming.

We introduce the vectors of ``marginals`` $\mu^i(\lambda)\in\primalspace,\ i\in\{1,2\}$ by
\begin{equation}\label{smoothDerivative_equ}
{\mu^i_{\rho}(\lambda)}_{w,x_w}:=\frac{\sum\limits_{x' \in \SX,x'_w=x_w}\exp\lan-\theta^i(\lambda)/\rho,\phi(x')\ran}{\exp(-\hat U^i_{\rho}(\lambda)/\rho)},\ w\in\SV^i\cup\SE^i\,.
\end{equation}
It is well-known, that the gradient of $\hat U_{\rho}$ is equal to 
\begin{equation}\label{gradientU_equ}
 \nabla\hat U_{\rho}(\lambda)=\mu^1_{\rho}(\lambda)_{\SV}-\mu^2_{\rho}(\lambda)_{\SV}\, .
\end{equation}
We refer to \citet[Lemma 1]{Savchynskyy11} for technical details.

\subsubsection{Tree-Reweighted Free Energy}
Let $N_v$ and $N_{uv}$ be the numbers of subgraphs containing node $v\in\SV$ and edge $uv\in\SE$ of the graph $\SG$. In the considered special case of the grid graph, $N_v=2$ and $N_{uv}=1$.

\begin{definition}
The function $\hat E_{\rho}\colon \SL\to \BR$ depending on a positive parameter $\rho$ 
\begin{multline}\label{smoothPrimal:equ}
\hat E_{\rho}(\mu):=\lan\theta,\mu\ran - \rho\Big(\sum_{v\in\SV}\sum_{x_v\in\SX_v}N_v\mu_{v,x_v}\log\mu_{v,x_v}\\ +\sum_{uv\in\SE}\sum_{x_{uv}\in\SX_{uv}}N_{uv}\mu_{uv,x_{uv}}\log\frac{\mu_{uv,x_{uv}}}{\mu_{v,x_v}\mu_{u,x_u}}\Big)
\end{multline}
is called the negative {\em tree-reweighted free energy}, introduced by~\cite{WainwrightBook08}.
\end{definition}

The problem of minimizing $\hat E_{\rho}$ on $\SL$ is important due to the fact that it is dual to the problem of maximizing $\hat U$, as shown \eg\ by~\cite{WernerReport09}. The duality holds not only in the considered special case of decomposition into two acyclic subgraphs, but for a decomposition into any number of acyclic subgraphs. In contrast to $U$, whose dual~\equ{LPRelaxation_equ} does not depend on the decomposition, the function $\hat E_{\rho}$ does, indeed. 
Maximizing $\hat U_{\rho}$ with $\rho=1$ is used by different algorithms \citep{WainwrightLogPartition2005,Jancsary2011d} to estimate marginal probabilities of the underlying Gibbs distribution \citep{WainwrightBook08}. The maximization delivers in the limit the same value of the objective as the minimization of $\hat E_{\rho}$. Hence it is important to compute {\em feasible} primal estimates based on dual iterates.

Another important meaning 
of $\hat E_{\rho}$ is considering it as an approximation of the relaxed energy $E$. This is due to the fact that 
the difference between~\equ{smoothPrimal:equ} and~\equ{localPolytopeEnergy_equ} vanishes continuously with $\rho$. This is stated precisely by the following lemma: 
\begin{lemma}\label{smoothVSnonsmoothPrimal_lem}    
For functions $E$ and $\hat E_{\rho}$ defined respectively by~\equ{localPolytopeEnergy_equ} and~\equ{smoothPrimal:equ} such a constant $C_{H}\ge 0$ exists that
\begin{equation}\label{smoothVSnonsmoothPrimal_equ}
 \hat E_{\rho}(\mu) \le E(\mu) \le \hat E_{\rho}(\mu) + \rho\cdot C_{H},\ \mu\in\SL\,
\end{equation}
and
\begin{equation}\label{smoothVSnonsmoothPrimal_equ_opt}
\hat E^*_{\rho} \le  E^* \le \hat E^*_{\rho}  + \rho\cdot C_{H} \,
\end{equation}
\end{lemma}
hold.
\begin{proof}
The value in brackets in~\equ{smoothPrimal:equ} can be represented as a convex combination of entropies of Gibbs distributions associated with subgraphs $\SG^i$ participating in the decomposition~\equ{decomposition_equ} up to a scale factor \citep{WainwrightLogPartition2005,WernerReport09}. The entropies are non-negative and bounded on $\SL$ functions. Their convex combination is non-negative and bounded as well, which proves the statement of the lemma.
\end{proof}
We will employ Lemma~\ref{smoothVSnonsmoothPrimal_lem} in Section~\ref{applications_sec}.

\paragraph{Optimizing Projection w.r.t.\ Tree-Reweighted Free Energy} The negative free energy $\hat E$ is separable w.r.t.\ $\mu_{\SV}$ and $\mu_{\SE}$ like the MAP-energy function $E$. Since the underlying constraint set --- the local polytope --- is the same, the definition of the optimizing projection w.r.t.\ the negative free energy $\hat E_{\rho}$  differs only slightly from the one for the MAP-energy $E$. Namely, the optimizing projection $\SP_{\hat E_{\rho},\SL}\colon \BR(\BM_{\SV})\times\BR(\BM_{\SE})\to\SL$ maps $(\mu_{\SV},\mu_{\SE})$ to $(\mu'_{\SV},\mu'_{\SE})$, defined as
\begin{align}
 &\ \mu'_v\ \ =\Pi_{\Delta(|\SX_v|)}(\mu_v),\ v\in\SV\label{primalOptMappingSimplexProjectionSmooth_equ}\,,\\
 &\begin{array}{ll}
 \mu'_{uv} = \argmin\limits_{\mu_{uv}\ge 0} &\lan \theta_{uv} + \rho N_{uv} \log\left(\mu'_u\mu'_v\right),\mu_{uv}\ran -\rho N_{uv}\lan \mu_{uv},\log\mu_{uv}\ran\\
	   & \mbox{s.t.}\,  A_{uv}\mu_{uv}=\mu'_v
\end{array},\ uv\in\SE\,. \label{primalOptMappingTransportpProblemSmooth_equ}
\end{align}
 
The only difference of $\SP_{\hat E_{\rho},\SL}$ to $\SP_{E,\SL}$ defined by~\equ{primalOptMappingSimplexProjection_equ}-\equ{primalOptMappingTransportpProblem_equ} is the objective function. It is not linear anymore as in~\equ{primalOptMappingTransportpProblem_equ}, but~\equ{primalOptMappingTransportpProblemSmooth_equ} constitutes rather a small-sized {\em entropy minimization problem}, which can be solved \eg\ by interior point methods. 

As before, we apply Theorem~\ref{mainPropertyOptProjection_tm} and Lemma~\ref{entropyFunctionLipschitzConst_lem} to get an idea about the convergence rate of the projected sequence $\SP_{\hat E_{\rho},\SL}(\mu^t_{\SV},\mu^t_{\SE}),\ t=1,\dots,\infty$, in comparison to the convergence rate of the original, infeasible sequence $\mu^t$. According to Lemma~\ref{entropyFunctionLipschitzConst_lem}, the estimation of the Lipschitz constant for the function~\equ{primalOptMappingTransportpProblemSmooth_equ} becomes bad (the constant becomes big) when some coordinate $\mu^t_{uv,x_{uv}}$ vanishes. However, the increase of the Lipschitz constant is only logarithmic w.r.t.\ the precision $\eps$ which has to be attained. Hence, the role of entropy terms $\mu^t_{uv,x_{uv}}\log\mu_{uv,x_{uv}}$ in slowing down the convergence of feasible primal estimates is typically insignificant.

\section{Application to Algorithmic Schemes}\label{applications_sec}

In previous sections we concentrated on the way to compute the optimizing projection assuming that the weakly converging (but infeasible) sequence is given. In contrast, this section is devoted to the methods of obtaining such sequences within different optimization algorithms.

\paragraph{Saddle-point Primal-Dual Algorithms}

In the simplest case the (infeasible) optimizing sequences for the primal~\equ{PrimalLP:equ} and dual~\equ{DualLP:equ} problems are generated by an algorithm itself, as it is typical for primal-dual saddle-point formulation based algorithms. A striking example of such an approach is the First Order Primal-Dual Algorithm (FPD) by~\cite{ChambollePock11}, which was recently applied to the local polytope relaxation of the MAP problem by~\cite{SchmidtEMMCVPR11}. 

The pair~\equ{PrimalLP:equ}-\equ{DualLP:equ} is cast as a saddle point problem via their Lagrangian,
\begin{equation}
      \label{eq:fpd-saddlepoint-lp}
      \max_{\mu \geq 0} 
      \min_{\duals} 
      \quad
      \lbrace
      \la -b, \duals \ra 
      + \la \mu , A^\T \duals \ra
      - \la \cost, \mu  \ra 
      \rbrace
      \,.
\end{equation}

The FPD algorithm iteratively updates the primal $\mu^t\in\primalspace$ and dual $\duals^t\in\dualspace$ approximate solutions and guarantees their weak convergence to the optimum of the primal~\equ{PrimalLP:equ} and dual~\equ{DualLP:equ} problems, respectively. However, the iterates $\mu^t$ and $\duals^t$ are not feasible in general and hence computing a duality gap requires their projection to the feasible sets. We do this by computing $\SP_{E,\SL}(\mu_{\SV}^t,\mu_{\SE}^t)$, as defined by~\equ{primalOptMappingSimplexProjection_equ}-\equ{primalOptMappingTransportpProblem_equ}, and $\SP_{\SU,\BD}(\duals^t_{\to},\duals^t_{0})$, defined by~\equ{trivialProjectionDual_equ}-\equ{trivialProjectionDual3_equ}.

\paragraph{Subgradient Ascent}

One of the first optimization algorithms with convergence guarantees for the dual decomposition based objective $U$ defined in~\equ{DDdual_equ} was 
subgradient ascent proposed by~\cite{SchlGig_12_usim2007} and~\cite{Komodakis07mrfoptimization}. It produces the sequence 
\begin{equation}\label{subgradientAlg_def}
 \lambda^{t+1}=\lambda^t + \tau^t \frac{\partial U}{\partial \lambda}(\lambda^t)\,,
\end{equation}
where $\tau^t$ is a positive step-size fulfilling the conditions
\begin{equation}
 \tau^t \to 0\,,\quad \sum_{t=1}^{\infty}\tau^t = \infty\,. 
\end{equation}

It is shown by~\cite{Larsson99_ergodic} and later applied by~\cite{Komodakis10journal} that time-averaged optimal labelings $\phi(x^{*i,t}),\ i=1,2$ (see~\equ{subgradient_equ}) converge to the primal solution of the relaxed MAP-inference problem~\equ{LPRelaxation_equ}. This implies that there is an optimal solution $\mu^*$ of~\equ{LPRelaxation_equ}, such that 
\begin{equation}\label{averageSubgrad_equ}
 \forall i: w\in\SV^i\cup\SE^i\quad \frac{\sum_{k=1}^{t}\phi_w(x^{*i,k})}{t}\xrightarrow{t\to\infty} \mu^*_w \,.
\end{equation}

The same convergence guarantee holds also for weighted averaging with step sizes $\tau^t$
\begin{equation}\label{averageWeightedSubgrad_equ}
  \forall i: w\in\SV^i\cup\SE^i\quad \frac{\sum_{k=1}^{t}\tau^k\phi_w(x^{*i,k})}{\sum_{k=1}^{t}\tau^k}\xrightarrow{t\to\infty} \mu^*_w \,.
\end{equation}

None of the sequences $\mu^t$, neither the one defined in~\equ{averageSubgrad_equ} nor that of~\equ{averageWeightedSubgrad_equ}, is feasible in general. Hence one has to apply the optimizing projection $\SP_{E,\SL}$ defined by~\equ{primalOptMappingSimplexProjection_equ}-\equ{primalOptMappingTransportpProblem_equ} to make them feasible. 
 Please note that an explicit form of $\mu_{\SE}^t$ is not important for this operation, since $\SP_{E,\SL}$ does not actually depend on it. Taking into account that the dual variables $\lambda^t$ are unconstrained and hence feasible without any projection, one can directly estimate the duality gap as $E(\SP_{E,\SL}(\mu_{\SV}^t,\mu_{\SE}^t))-U(\lambda^t)$.

\paragraph{Methods Based on Smoothing, Tree-Reweighted Primal Bound \citep{Savchynskyy11,SavchynskyyUAI2012,Jancsary2011d,HazanShashua}.} Reconstructing a primal sequence for gradient-based methods optimizing the smoothing objective $\hat U$ is similar to~\equ{averageSubgrad_equ}, but does not even require an averaging of gradients over iterations to obtain convergence. 

In particular, the vector $\mu_{\rho}^i(\lambda^t)$ defined by~\equ{smoothDerivative_equ} converges to an optimum $\hat \mu^*$ of the tree-reweighted free energy $\hat E_{\rho}$ as $\lambda^t$ approaches an optimum $\hat \lambda^*$ of the smoothed dual function $\hat U_{\rho}$:
\begin{equation}\label{SmoothGradConverg_equ}
  \forall i: w\in\SV^i\cup\SE^i\quad \mu_{\rho}^i(\lambda^t)\xrightarrow{\lambda^t\to\hat \lambda^*} \hat \mu^*_w \,.
\end{equation}

Analogous to the subgradient optimization we can apply the optimizing mapping $\SP_{\hat E_{\rho},\SL}$ to get a feasible primal estimate without knowing an explicit expression for $\mu_{\SE}^t$, since the optimizing projection does not depend on its second argument. As in the non-smooth case, since the $\lambda^t$ are unconstrained, the value $\hat E_{\rho}(\SP_{\hat E_{\rho},\SL}(\mu_{\SV}^t,\mu_{\SE}^t))-\hat U_{\rho}(\lambda^t)$ is the duality gap.

\begin{remark}
If the final objective of the optimization is not the tree-reweighted free energy $\hat E_{\rho}$, but the MAP-energy $E$, and the smoothing is used as an optimization tool to speed up convergence, one can obtain even better primal bounds for a lesser computational cost. Namely, due to~\equ{smoothVSnonsmoothPrimal_equ} the optimizing projection $\SP_{E,\SL}$ can be applied to approximate the optimal solution of the MAP-energy $E$. Denote
\begin{equation}
\hat\mu'=(\hat\mu'_{\SV},\hat\mu'_{\SE})= \SP_{\hat E_{\rho},\SL}(\mu_{\SV},\mu_{\SE})
\quad
\mbox{and} 
\quad
\mu'=(\mu'_{\SV},\mu'_{\SE})= \SP_{E,\SL}(\mu_{\SV},\mu_{\SE})\,.
\end{equation}

From the definitions~\equ{primalOptMappingSimplexProjection_equ} and~\equ{primalOptMappingSimplexProjectionSmooth_equ} follows that  $\hat\mu'_{\SV}=\mu'_{\SV}$, and thus due to~\equ{primalOptMappingTransportpProblem_equ} and~\equ{primalOptMappingTransportpProblemSmooth_equ} $E(\mu')\le E(\hat\mu')$. This means that the projection $\SP_{E,\SL}$ is preferable for approximating the minimum of $E$ over $\SL$ even in the case when the smooth objective $\hat U_{\rho}$ was optimized and not the original non-smooth $U$. As an additional benefit, one obtains faster convergence of the projection even from the wost-case analysis, due to a better estimates of the Lipschitz constant for the function $E$ compared to the function $\hat E_{\rho}$, as estimated in Lemmas~\ref{linearFunctionLipschitzConst_lem} and~\ref{entropyFunctionLipschitzConst_lem}.
\end{remark}

\paragraph{Bundle methods, ADLP, ADMM, TRWS, MPLP and others.} Analogous converging primal sequences can be constructed for other optimization approaches as well. For bundle methods, recently applied to MAP-inference by~\cite{Kappes2012}, one has to average the resulting optimal labelings $\phi(x^{*i,t}),\ t=1,\dots,\infty$, with weights $\xi_t$ obtained from the solution of the auxiliary problem, see~\citet[eq.~23]{Kappes2012}. 

Other examples are augmented Lagrangian based optimization schemes, which were recently applied to the MAP inference problem by~\cite{MartinsICML11,MeshiGloversonECML11}. These algorithms augment the Lagrangian~\equ{eq:fpd-saddlepoint-lp} with a quadratic term and combine coordinate descent in the primal domain with subgradient steps in the dual one or vice versa. An important property of these schemes is that they maintain (in general) infeasible primal and dual estimates, which can be projected to the feasible sets with optimizing projections similar to the described $\SP_{E,\SL}$ and $\SP_{\SU,\BD}$.

However, we are not aware of methods for reconstructing primal solutions of the relaxed problem from dual estimates for non-smooth coordinate descent based schemes like {TRW-S} by~\cite{Kolmogorov06} and MPLP by~\cite{GlobersonJakollaNIPS2007}. Indeed, these schemes do not solve the relaxed MAP problem in general, hence even if one would have such a method at hand, it would not guarantee convergence of the primal estimates to the optimum.

\section{Experimental Analysis and Evaluation}\label{demo_Section}

The main goal of this section is to show how Theorem~\ref{mainPropertyOptProjection_tm} works in practice. Hence we provide only two groups of experiments to evaluate our method. Both concentrate on reconstructing feasible primal estimates for the MAP inference algorithms considered in Section~\ref{applications_sec}. In the first group we show how the projected primal MAP-solution converges to the optimum for a series of algorithms. In the second one we show how the bound~\equ{mainPropertyOptProjection_equ}-\equ{mainPropertyOptProjectionTrivialX_equ} allows for at least qualitative prediction of the objective value in the (feasible) projected point. We refer to our conference papers \citep{Savchynskyy11,SchmidtEMMCVPR11,Kappes2012,SavchynskyyUAI2012} for the experiments with an extended set of benchmark data.

For the experiments we employ our own implementations of the First Order Primal Dual Algorithm (acronym {\bf FPD}) as described by~\cite{SchmidtEMMCVPR11}, the adaptive diminishing smoothing  algorithm {\bf ADSAL} proposed by~\cite{SavchynskyyUAI2012}, the dual decomposition based subgradient ascent with an adaptive step-size rule according to~\citet[eq.17]{Kappes2012} and primal estimates based on averaged~\equ{averageSubgrad_equ} (acronym {\bf SG-AVE}) and weighted averaged~\equ{averageWeightedSubgrad_equ} (acronym {\bf SG-WEI}) subgradients, and finally Nesterov's accelerated gradient ascent method applied to the smoothed dual decomposition based objective~\equ{smoothApprox_ineq} (acronym {\bf NEST}) studied by~\cite{Savchynskyy11}. All implementations are based on data structures of the OpenGM library by \cite{KappesOpenGM2Arxiv}.

The optimizing projection to the local polytope w.r.t.\ to the MRF energy~\equ{primalOptMappingSimplexProjection_equ}-\equ{primalOptMappingTransportpProblem_equ} is computed using our implementation of a specialization of the simplex algorithm for transportation problems \citep{BazaraaJarvis77}. We adopted an elegant method by~\cite{BlandAnticycling76}, also discussed by~\cite{PapadimitriuSteiglitz}, to avoid cycling. The source code of the solver can be downloaded from the first author's web-site\footnote{http://hci.iwr.uni-heidelberg.de/Staff/bsavchyn/software.php}. 

\paragraph{Feasible Primal Bound Estimation} In the first series, we demonstrate that for all three groups of methods discussed in Section~\ref{applications_sec} our method efficiently provides feasible primal estimates for the MAP inference problem~\equ{LPRelaxation_equ}. To this end we generated a $256\times 256$ grid model with $4$ variable states ($|\SX_v|=4$) and potentials randomly distributed in the interval $[0,1]$. We solved an LP relaxation of the MAP inference problem with {\bf FPD} as a representative of methods dealing with infeasible primal estimates, subgradient methods {\bf SG-AVE}, {\bf SG-WEI} and {\bf ADSAL} as the fastest representatives of smoothing-based algorithms. The corresponding plots are presented in Fig.~\ref{fig:exp1}. We note that in {\em all} experiments the time needed to compute the optimizing projection $\SP_{E,\SL}$ did not exceed the time needed to compute the subgradient/gradient of the dual function $U$/$\hat U_{\rho}$ and requires 0.01-0.02 s on a 3GHz machine. The generated dataset is not LP tight, hence the obtained relaxed primal solution has a significantly smaller energy than the integer one. In contrast to the cases where only non-relaxed integer primal estimates are computed, the primal and dual bounds of the relaxed problem converge to the same limit value. Due to the feasibility of both primal and dual estimates, the primal and dual objective functions' values bound the optimal value of the relaxed problem from above and below, respectively.

\begin{figure}[t]
\begin{center}
\mbox{\includegraphics[width=0.48\linewidth]{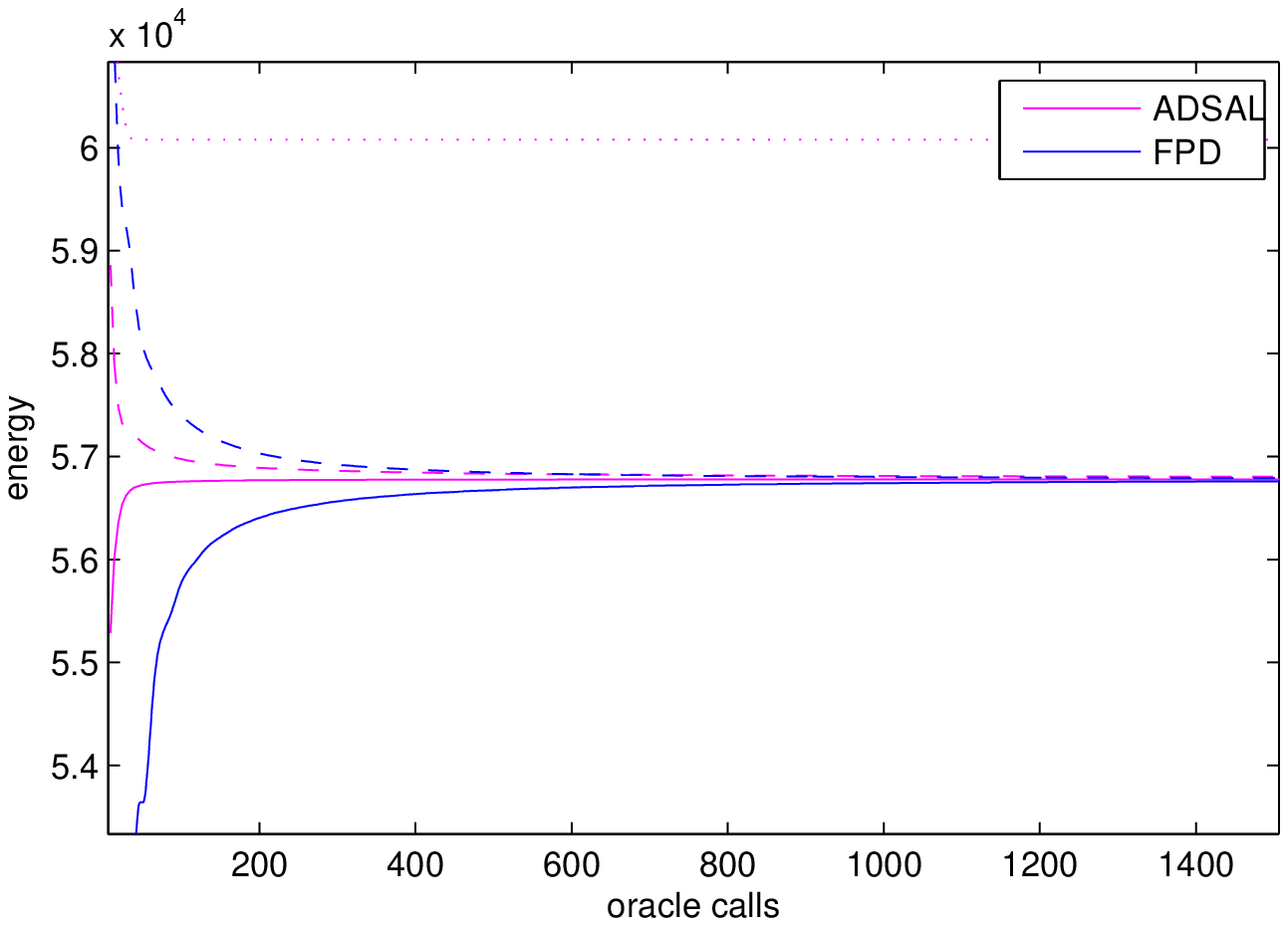}}
\mbox{\includegraphics[width=0.48\linewidth]{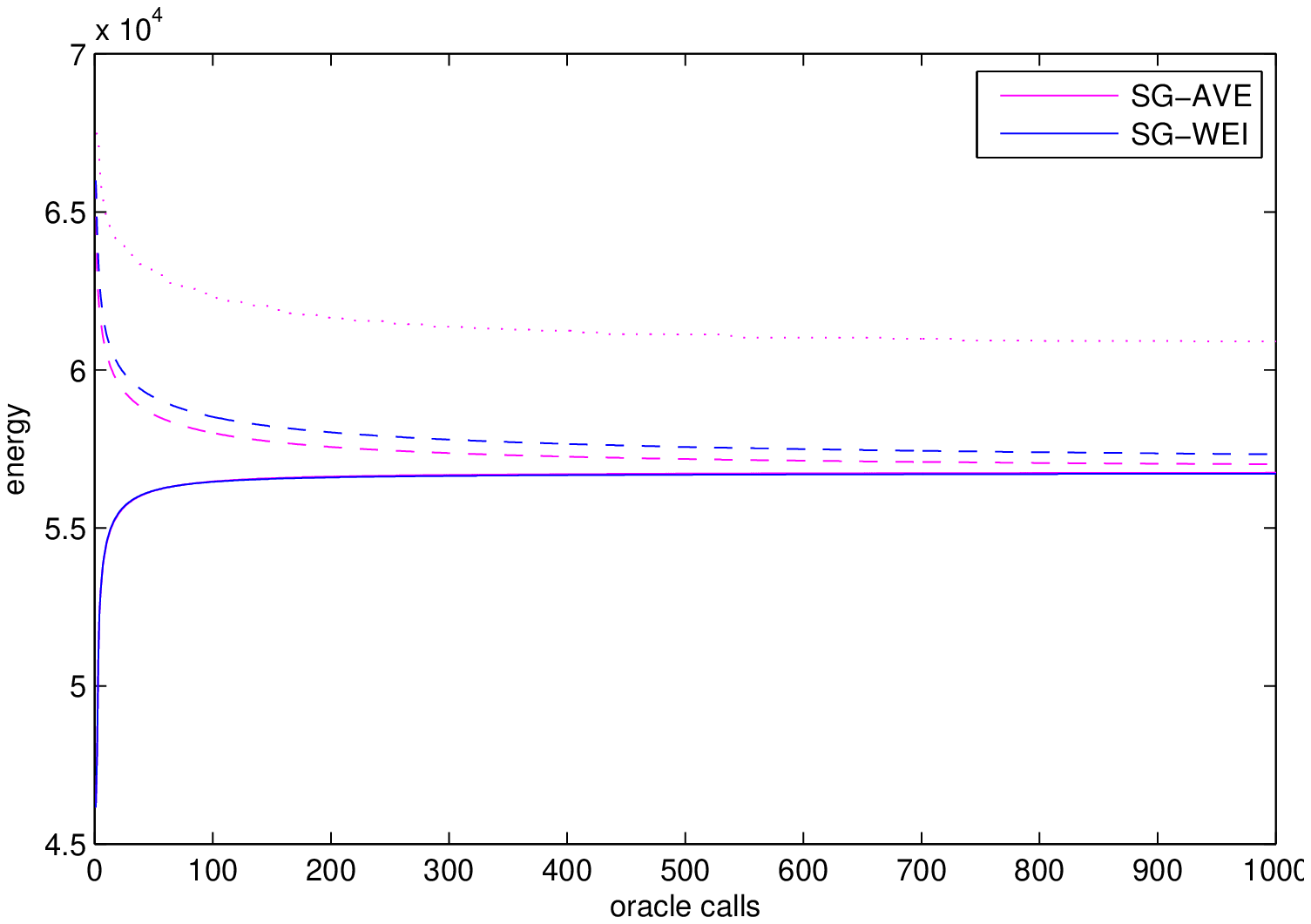}}
\end{center}

\caption{
  Convergence of the primal (dashed lines) and dual (solid lines) bounds to the same optimal limit value for {\bf ADSAL} and {\bf FPD} algorithms (left) and {\bf SG-AVE} and {\bf SG-WEI} (right). The obtained integer bound is plotted as a dotted line. Note that due to the feasibility of both primal and dual estimates, the primal and dual objective functions' values bound the optimal value of the relaxed problem from above and below, respectively.
}
\label{fig:exp1}
\end{figure}

\paragraph{Evaluation of Convergence Bound} The second series of experiments is devoted to the evaluation of the convergence bounds provided by Theorem~\ref{mainPropertyOptProjection_tm}. To this end, we generated four LP-tight grid-structured datasets with known optimal labeling. We refer to~\citet[pp.~95-96]{SchmidtEMMCVPR11} for a description of the generation process. The resulting unary and pairwise potentials were distributed in the interval $[-10,10]$. We picked up a random subset of edges not belonging to the optimal labeling and assigned them ``infinite'' values. We created four datasets with ``infinities'' equal to $10~000$, $100~000$, $1~000~000$ and $10~000~000$ and ran {\bf NEST} for inference. According to Theorem~\ref{mainPropertyOptProjection_tm} the energy $E$ evaluated on projected feasible estimates $\SP_{E,\SL}(\mu_{\SV}^t,\mu_{\SE}^t)$, $t=1,\dots,\infty$, where the $\mu^t$ were reconstructed from dual estimates according to~\equ{SmoothGradConverg_equ}, can be represented as
\begin{equation}
 E(\SP_{E,\SL}(\mu_{\SV}^t,\mu_{\SE}^t))=F(\mu^t) + L_Y(E)\Vert\mu^t - \Pi_{\SL}\mu^t\Vert
\end{equation}
for a suitably selected function $F$. Since {\bf NEST} is a purely dual method and ``infinite'' pairwise potentials did not make any significant contribution to  values and gradients of the (smoothed) dual objective, the infeasible primal estimates $\mu^t$ were the same for all four different approximations of the infinity value. Since according to Lemma~\ref{linearFunctionLipschitzConst_lem} the Lipschitz constant $L_Y(E)$ is asymptotically proportional to the values of the binary potentials $\theta_{\SE}$ we plotted the values $\log E(\SP_{E,\SL}(\mu_{\SV}^t,\mu_{\SE}^t))$ as a function of $t$ for all four datasets in Fig.~\ref{fig:exp2}. As predicted by Theorem~\ref{mainPropertyOptProjection_tm} the corresponding energy values differ by approximately a factor of $10$, as the ``infinite'' values do. Due to the logarithmic energy scale this difference corresponds to equal log-energy distances between the curves in Fig~\ref{fig:exp2}.

\begin{figure}[t]
\begin{center}
\mbox{\includegraphics[width=0.48\linewidth]{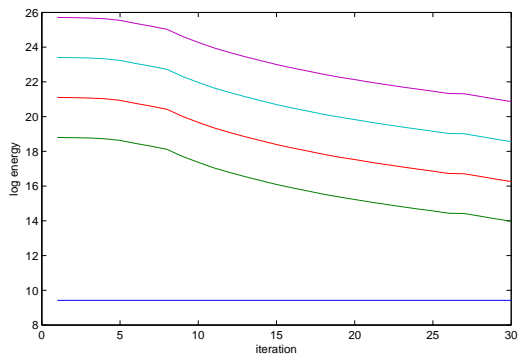}}
\end{center}

\caption{
  Convergence of the obtained primal feasible solution for four datasets which differ only by the values used as ``infinity''. The energy values are plotted in logarithmic scale. From bottom to top: optimal log-energy, primal bounds corresponding to infinity values equal to $10~000$, $100~000$, $1~000~000$ and $10~000~000$. Please note that as predicted by Theorem~\ref{mainPropertyOptProjection_tm} and Lemma~\ref{linearFunctionLipschitzConst_lem} the distance between corresponding log-energies remains approximately the same for all time steps and is equal to $\log 10$, which corresponds to the multiplication factor determining the relation between different values of ``infinity''. 
}
\label{fig:exp2}
\end{figure}

%
 
\section{Conclusions}\label{Conclusions_Section}

We presented an efficient and quite general optimizing projection method for computing feasible primal estimates for dual and primal-dual optimization schemes. The method provides convergence guarantees similar to the ones of the Euclidean projection, but contrary to it, it allows for efficient computations, when the feasible set and the objective function posses certain separability properties.
As any optimization tool it has also certain limitations related to the Lipschitz continuity of the primal objective, however exactly the same limitations are characteristic also for the Euclidean projection. Hence they can not be considered as disadvantages of particularly this method, but rather as disadvantages of all projection methods in general and can be overcome only by constructing algorithms, which intrinsically maintain feasible primal estimates during iterations. The construction of such algorithms has to be addressed in future work.

\paragraph{Acknowledgement.} This work has been supported by the German Research Foundation (DFG) within the program ``Spatio-/Temporal Graphical Models and Applications in
Image Analysis'', grant GRK 1653.

\bibliographystyle{named}
\bibliography{main}
\end{document}